\documentclass[11pt]{amsart}

\numberwithin{equation}{section}
\usepackage{amssymb,amsmath}
\usepackage{amsmath,amsfonts,amsthm,mathrsfs,a4wide}
\usepackage{color,graphics,epsfig}

\newcommand{\calA}{\mathcal{A}}

\newcommand{\calG}{\mathcal{G}}

\newcommand{\calL}{\mathcal{L}}

\newcommand{\mC}{\mathbb{C}}

\newcommand{\mS}{\mathbb{S}}
\newcommand{\mT}{\mathbb{T}}

\newtheorem{theorem}{Theorem}[section]
\newtheorem{lemma}[theorem]{Lemma}

\newtheorem{proposition}[theorem]{Proposition}

\theoremstyle{definition}

\theoremstyle{definition}
\newtheorem{definition}[theorem]{Definition}

\theoremstyle{definition}
\newtheorem{notation}[theorem]{Notation}

\begin{document}



\title[Riccati equations in Banach algebras]{Solvability of the $H^\infty$
  algebraic Riccati equation in Banach algebras}

\author{Amol Sasane}
\address{Department of Mathematics, London School of Economics,
     Houghton Street, London WC2A 2AE, United Kingdom.}
\email{A.J.Sasane@lse.ac.uk}

\begin{abstract}
  Let $R$ be a commutative complex unital semisimple Banach algebra
  with the involution $\cdot ^\star$. Sufficient conditions are given
  for the existence of a stabilizing solution to the $H^\infty$
  Riccati equation when the matricial data has entries from $R$.
  Applications to spatially distributed systems are discussed.
\end{abstract}

\maketitle

\section{Introduction}

In the standard (regular full information) $H^\infty$ control problem,
one wants to find an appropriate control input $u$ such that the
effect of the disturbance $w$ on the output $z$ is minimized for the
control system given by the following dynamics:
\begin{eqnarray}
\label{eq_system_a}
x'(t)&=& Ax(t)+Bu(t)+Ew(t),\\ 
\label{eq_system_b}
z(t) &=& Cx(t)+D_1u(t)+D_2w(t),
\end{eqnarray}
where $x(t)\in \mC^n$, $u(t)\in \mC^m$, $w(t)\in \mC^\ell$, $z(t)\in
\mC^p$, and $A,B,E,C,D_1,D_2$ are complex matrices of appropriate
sizes. More precisely, one seeks a feedback law
$$
u(t)=F x(t)
$$
such that the closed loop system obtained in this manner is internally
stable and its transfer function has $H^\infty$ norm strictly less
than some a priori given bound $\gamma>0$. The following result is
well-known; see for example \cite[Theorem~3.1, p.49]{Sto}:

\begin{theorem}
  Consider \eqref{eq_system_a}-\eqref{eq_system_b}, and let
  $\gamma>0$. Suppose that the tuple $(A,B,C,D_1)$ has no invariant
  zeros on the imaginary axis, and that $\ker D_1=\{0\}$. Then the
  following are equivalent:
\begin{enumerate}
\item There exists a matrix $F$ such that the closed loop system
  obtained by applying the feedback $u(t)=F x(t)$ is
  internally stable, and the $H^\infty$ norm of the closed loop
  transfer function is strictly less than $\gamma$.
\item $D_2^* D_2 <\gamma^2 I$. Moreover,
  there exists a positive semi-definite solution of the algebraic
  Riccati equation
\begin{equation}
\label{eq_riccati_classical}
0=A^* P+ PA+C^*C-\left[\begin{array}{cc} B^*P+D_1^* C\\E^*P+D_2^* C \end{array}\right]^* 
\left[\begin{array}{cc} D_1^*D_1 & D_1^* D_2 \\ D_2^* D_1 & D_2^* D_2 -\gamma^2 I \end{array}\right]^{-1} 
\left[\begin{array}{cc} B^*P+D_1^* C\\E^*P+D_2^* C \end{array}\right],
\end{equation}
such that $A_{\textrm{cl}}$ is exponentially stable, that is, $\textrm{\em
  Re}(\lambda)<0$ for all eigenvalues $\lambda$ of $ A_{\textrm{cl}}$, where 
\begin{equation}
\label{eq_closed_loop_A}
A_{\textrm{\em cl}}:= A-\left[\begin{array}{cc} B & E
  \end{array}\right]\left[\begin{array}{cc} D_1^*D_1 & D_1^* D_2 \\
    D_2^* D_1 & D_2^* D_2 -\gamma^2 I \end{array}\right]^{-1}
\left[\begin{array}{cc} B^*P+D_1^* C\\E^*P+D_2^* C \end{array}\right].
\end{equation}
\end{enumerate}
If $P$ satisfies the conditions in (2) above, then the matrix $F$
satisfying (1) can be taken as
$$
F:= - (D_1^* (I-\gamma^{-2} D_2 D_2^*)^{-1} D_1)^{-1} (D_1^*
C+B^*P+D_1^* D_2 (\gamma^2 I-D_2^* D_2)^{-1} (D_2^* C+E^*P)).
$$
\end{theorem}

The interest in systems over rings has recently been revived owing to
potential applications for control theoretic problems of spatially
invariant systems \cite{BamPagDah}.  In particular, a natural question
that arises is the following: if the matricial data in the algebraic
Riccati equation has entries from a Banach algebra $R$, then does
there exist a solution matrix $P$ which also has entries from $R$?  In
the context of the algebraic Riccati equation associated with the
problem of optimal control for a linear control system with a
quadratic cost (LQ problem), this was done recently in \cite{CurSas}.
In this article, we continue our study, and show that an analogous
result can also be established for the $H^\infty$ algebraic Riccati
equation. Although the method of proof is identical to the one
followed in \cite{CurSas}, the result in the current article does not
follow from \cite{CurSas}, and it is not automatic. There are salient
differences in the proof:
\begin{enumerate}
\item One of the key differences is that the quadratic term of the
  form $P^*QP$ in the algebraic Riccati equation has a {\em positive
    semi-definite} matrix $Q$ in the LQ problem, whereas the matrix
  $Q$ is {\em indefinite} in the $H^\infty$ control problem. In
  particular, this manifests in that we need to prove the technical
  result (Proposition~\ref{prop2}) on the continuous dependence of the
  stabilizing solution of the $H^\infty$ algebraic Riccati equation on
  the data.
\item The second new feature is the presence of the {\em  inverse} of a
  matrix with feedthrough terms $D_1,D_2$ in the $H^\infty$
  algebraic Riccati equation. This results in an extra complication in
  the application of the inverse function theorem for Banach algebras,
  namely, the application of Proposition~\ref{prop3}.
\end{enumerate}

\noindent We begin by fixing some notation.

\begin{notation}
  Throughout the article, $R$ will denote a commutative, unital,
  complex, semisimple Banach algebra, which possesses an involution
  $\cdot^\star$. 

  On the other hand, the usual adjoint of a matrix $M=[m_{ij}]\in
  \mC^{p\times m}$ will be denoted by $M^* \in \mC^{m\times p}$, that
  is, $ M^* =[\overline{m_{ji}}]$.

$M(R)$ will denote the maximal ideal space of $R$, equipped with the
weak-$\ast$ topology. For $x\in R$, we will denote its Gelfand
transform by $\widehat{x}$, that is, 
$$
\widehat{x}(\varphi)=\varphi(x), \; \varphi \in M(R),\; x\in R.
$$
For a matrix $M\in R^{p\times m}$, whose entry in the $i$th row and
$j$th column is denoted by $m_{ij}$, we define $M^\star\in R^{m\times
  p}$ to be the matrix whose entry in the $i$th row and $j$th column
is $m_{ji}^\star$. 

Also by $\widehat{M}$ we mean the $p\times m$
matrix, whose entry in the $i$th row and $j$th column is the
continuous function $\widehat{m_{ij}}$ on $M(R)$. Summarizing, if
$M=[m_{ij}]\in R^{p\times m}$, then
\begin{eqnarray*}
M^\star 
&=&  
\Big[m_{ji}^\star \Big]\in R^{m\times p},
\\
\widehat{M} 
&=& 
\Big[\widehat{m_{ij}}\Big]\in \Big(C(M(R);\mC)\Big)^{p\times m},
\\
\Big(\widehat{M}(\varphi)\Big)^* 
&=&
\Big[\overline{\widehat{m_{ji}}(\varphi)}\Big]\in \mC^{p\times m}.
\end{eqnarray*}
\end{notation}

Our main result is the following. 

\begin{theorem}
\label{main_theorem}
Let $A\in R^{n\times n}$, $B\in R^{n\times m}$, $C\in R^{p\times n}$,
$D_1\in R^{p\times m}$, $D_2 \in R^{p\times \ell}$, $E\in R^{n\times
  \ell}$. Suppose that for all $\varphi\in M(R)$,
\begin{itemize}
\item[(A1)] $\widehat{(A^\star)}(\varphi)= (\widehat{A}(\varphi))^*$,
  $\widehat{(B^\star)}(\varphi)= (\widehat{B}(\varphi))^*$,
  $\widehat{(C^\star)}(\varphi)= (\widehat{C}(\varphi))^*$,
  $\widehat{(D_1^\star)}(\varphi)= (\widehat{D_1}(\varphi))^*$,
  $\widehat{(D_2^\star)}(\varphi)= (\widehat{D_2}(\varphi))^*$,
  $\widehat{(E^\star)}(\varphi)= (\widehat{E}(\varphi))^*$, 
\item[(A2)]
  $(\widehat{A}(\varphi),\widehat{B}(\varphi),\widehat{C}(\varphi),\widehat{D_1}(\varphi))$
  is left invertible,
\item[(A3)] there exist matrices $F_1(\varphi),F_2(\varphi)$ such that
  $\widehat{A}(\varphi)+\widehat{B}(\varphi)F_1(\varphi)$ is
  exponentially stable, and  we have (suppressing the argument $\varphi$) 
$$
\|(\widehat{C}+\widehat{D_1}F_1)(sI-\widehat{A}-\widehat{B}F_1)^{-1}
(\widehat{E}+\widehat{B}F_2)+\widehat{D_2}+\widehat{D_1}F_2\|_\infty
<\gamma,
$$ 
\item[(A5)] 
  $(\widehat{A}(\varphi),\widehat{B}(\varphi),\widehat{C}(\varphi),\widehat{D_1}(\varphi))$
  has no invariant zeros on the imaginary axis, and 
\item[(A6)]  $\ker( \widehat{D_1}(\varphi))=\{0\}$.
\end{itemize}
Then there exists a $P\in R^{n\times n}$ such that 
\begin{enumerate}
\item $ 0=A^\star P+PA + C^\star C$

$\quad -\left[\begin{array}{cc} P B+C^\star D_1 & 
P E + C^\star D_2 \end{array}\right] 
\left[\begin{array}{cc} D_1^\star D_1 & D_1^\star D_2 \\ 
D_2^\star D_1 & D_2^\star D_2 -\gamma^2 I \end{array}\right]^{-1} 
\left[\begin{array}{cc} B^\star P+ D_1^\star C\\
E^\star P+ D_2^\star C \end{array}\right]$,
\item The matrix 
$$
A_{\textrm{\em cl}}:=A-\left[\begin{array}{cc} B & E
  \end{array}\right]\left[\begin{array}{cc} D_1^\star D_1 & D_1^\star D_2 \\
    D_2^\star D_1 & D_2^\star D_2 -\gamma^2 I \end{array}\right]^{-1}
\left[\begin{array}{cc} B^\star P+D_1^\star C\\E^\star P+D_2^\star C \end{array}\right]
$$ 
is exponentially stable, and
\item for all $\varphi\in M(R)$, $\widehat{P}(\varphi)$ is positive
  semidefinite. $\phantom{\left[\begin{array}{cc} D_1^\star \\
        D_2^\star \end{array}\right]}$
\end{enumerate}
\end{theorem}

\noindent In the following we define what is meant by ``exponentially stable''
in the conclusion (2) in the above result.

\begin{definition}
Let $R$ be a commutative, unital, complex, semisimple Banach 
algebra. If $A\in R^{n\times n}$, then let $M_A:R^n\rightarrow R^n $ be the
multiplication map by the matrix $A$, that is, $v\mapsto Av$ ($v\in
R^n$).  Then $R^{n\times n}$ is a unital complex Banach algebra (for
example) with the norm
$$
\|A\|:=\|M_A\|_{\calL(R^n)}\quad (A\in R^{n\times n})
$$
where $\calL(R^n)$ denotes the set of all continuous linear
transformations from $R^n$ to $R^n$, and $R^n$ is the Banach space
equipped (for example) with the norm
$$
\|x\|=\max\{\|x_k\|:1\leq k\leq n\} \textrm{ for } 
x=\left[\begin{array}{c} x_1\\\vdots\\ x_n\end{array}\right],
$$
and $\calL(R^n)$ is equipped with the usual operator norm:
$$
\|M_A\|_{\calL(R^n)}=\sup\{ \|Av\|:v\in R^n \textrm{ with }\|v\|\leq 1\}.
$$
For $A\in R^{n\times n}$, we define
$$
e^{A}=\sum_{k=0}^\infty \frac{1}{k!}A^k.
$$
The absolute convergence of this series is established just as in the
scalar case.

$A\in R^{n\times n}$ is said to be {\em exponentially stable} if there
exist positive constants $C$ and $\epsilon$ such that 
$$
\|e^{tA}\|\leq Ce^{-\epsilon t} \textrm{ for all } t \geq 0.
$$
\end{definition}

We recall the following two results \cite[Lemma~1.6, Prop.1.7]{CurSas}.

\begin{lemma}
\label{lemma_spectra}
Let $A\in R^{n\times n}$. Then the following are equivalent:
\begin{enumerate}
\item $\lambda$ belongs to the spectrum of $A\in R^{n\times n}$.
\item $\lambda$ belongs to the spectrum of $M_A\in \calL(R^n)$.
\item $\lambda$ belongs to the spectrum of $\widehat{A}(\varphi)$ for some $\varphi \in M(R)$. 
\end{enumerate}
\end{lemma}

\begin{proposition}
\label{prop_char_exp_stab}
Let $A\in R^{n\times n}$. Then $A$ is exponentially stable if and only if
$$
\sup \{ \textrm{\em Re}(\lambda):\lambda \textrm{ is an eigenvalue of }
\widehat{A}(\varphi) \textrm{ for some } \varphi \in M(R)\} <0.
$$
\end{proposition}

The proof in Section \ref{Proof} of our main result above is
established by taking the Gelfand transform of our equation, and
showing that the pointwise solution is continuous.  Then we use the
Banach algebra operational calculus to ensure that this continuous
solution is actually the Gelfand transform of a matrix with entries
from the Banach algebra.

In Section~\ref{Applications} we discuss the applications of this
result to the control of spatially invariant systems.

\section{Proof of the main result}\label{Proof}

We will need the following two results. The first one gives sufficient
conditions on the matricial data for the classical $H^\infty $
algebraic Riccati equation to have a positive semi-definite
stabilizing solution, and moreover it says the smallest such solution
$P$ depends continuously on the matricial data.

\begin{proposition}
\label{prop2}
Let $\mS$ be the set of matrix tuples $(A,B,C,D_1,D_2,E)$ belonging to 
$\mC^{n\times n}\times \mC^{n\times m}\times \mC^{p\times n}\times
\mC^{p\times m}\times \mC^{p\times \ell}\times \mC^{n\times \ell}$  
such that 
\begin{enumerate}
\item $(A,B,C,D_1)$ is left invertible,
\item there exist matrices 
$F_1,F_2$ such that $A+BF_1$ is exponentially stable, and 
$$
\|(C+DF_1)(sI-A-BF_1)^{-1} (E+BF_2)+D_2+D_1F_2\|_\infty <\gamma,
$$ 
\item  $(A,B,C,D_1)$ has no invariant zeros on the imaginary axis, and 
\item  $\ker D_1=\{0\}$. 
\end{enumerate}
Then there exists a smallest positive semidefinite, solution
$P=P(A,B,C,D_1,D_2,E)\in \mC^{n\times n}$ of the Riccati equation
\eqref{eq_riccati_classical} such that $A_{\textrm{cl}}$ given by
\eqref{eq_closed_loop_A} is exponentially stable.  Moreover, the map
$(A,B,C,D_1,D_2,E)\mapsto P(A,B,C,D_1,D_2,E):\mS\rightarrow
\mC^{n\times n}$ is continuous.
\end{proposition}
\begin{proof} The proof is the same as that of
  \cite[Lemma~3.1]{StoSCL}, mutatis mutandis, except we have complex
  matrices instead of real ones, and transposes there have to be
  replaced by the Hermitian adjoints. We give repeat the proof here.
  The existence and uniqueness of the stabilizing solution is
  well-known; see \cite{Sto}. The continuity follows in a
  straightforward manner. The solution $P$ is associated with the
  stable subspace of a Hamiltonian matrix . Since this Hamiltonian
  matrix has no eigenvalues on the imaginary axis, the eigenvalues in
  the open left half plane and the open right half plane are
  separated, and the existence of a continuous basis for the stable
  subspace and hence the continuous dependence of the stabilizing
  solution of the $H^\infty$ algebraic Riccati equation can be found
  for example in \cite{Ste}.
\end{proof}

The next result we will need is the following version of the Implicit
Function Theorem in the context of Banach algebras (see
\cite[p.155]{Hay}), and this will be used in proving our main results
when we need to pass from continuous functions on $M(R)$ to elements
of $R$.

\begin{proposition}
\label{prop3}
Let $h_1,\dots, h_s$ be continuous functions on $M(R)$. Suppose that
$f_1,\dots f_r$ in $R$ and $
G_1(z_1,\dots, z_{s+r}),\dots, G_t(z_1,\dots,z_{s+r})
$ are holomorphic functions 
with $t\geq s$ defined on a neighbourhood of the joint spectrum
$$
\sigma(h_1,\dots, h_s, f_1,\dots, f_r)
:=
\{ (h_1(\varphi),\dots,h_s(\varphi), \widehat{f_1}(\varphi),\dots,
\widehat{f_r}(\varphi)):\varphi\in M(R)\},
$$
such that 
\begin{equation}
\label{eq_hayashi_cont_soln}
G_k(h_1,\dots, h_s, \widehat{f_1},\dots,\widehat{f_r})=0 
\textrm{ on }M(R) \textrm{ for }1\leq k\leq t.
\end{equation}
If the rank of the Jacobi matrix $\displaystyle
\frac{\partial(G_1,\dots, G_t)}{\partial(z_1,\dots, z_s)}$ is $s$ on
$\sigma(h_1,\dots, h_s, f_1,\dots, f_r)$, then there exist elements
$g_1,\dots, g_s$ in $R$ such that
$$
\widehat{g_1}=h_1,\dots,\widehat{g_s}=h_s .
$$
\end{proposition}

We are now ready to prove our main result. 

\begin{proof}[Proof of Theorem~\ref{main_theorem}] If we fix a
  $\varphi\in M(R)$, then owing to the assumptions (A2)-(A6), we know
  that there is a unique smallest positive semidefinite solution in
  $\mC^{n\times n }$, which we will denote by $\Pi(\varphi)$, such
  that (suppressing the argument $\varphi$) 
\begin{eqnarray}
\label{riccati_1}
\quad \quad 0&=&(\widehat{A})^*\Pi+\Pi\widehat{A} + (\widehat{C})^*\widehat{C}
\\
\nonumber 
&&-\left[\begin{array}{cc} (\widehat{B})^*\Pi+(\widehat{D_1})^* \widehat{C}\\
(\widehat{E})^*\Pi +(\widehat{D_2})^* \widehat{C}\end{array}\right]^* 
\left[\begin{array}{cc} (\widehat{D_1})^*\widehat{D_1} & (\widehat{D_1})^* \widehat{D_2} \\ 
(\widehat{D_2})^* \widehat{D_1} & (\widehat{D_2})^* \widehat{D_2} -\gamma^2 I \end{array}\right]^{-1} 
\left[\begin{array}{cc} (\widehat{B})^*\Pi+(\widehat{D_1})^* \widehat{C}\\
(\widehat{E})^*\Pi+(\widehat{D_2})^* \widehat{C} \end{array}\right]
\end{eqnarray}
and 
$$
\calA_{\textrm{cl}}:=\widehat{A}-\left[\begin{array}{cc} \widehat{B} & \widehat{E}
  \end{array}\right]\left[\begin{array}{cc} (\widehat{D_1})^*\widehat{D_1}& (\widehat{D_1})^* \widehat{D_2} \\
    (\widehat{D_2})^* \widehat{D_1} & (\widehat{D_2})^* \widehat{D_2} -\gamma^2 I \end{array}\right]^{-1}
\left[\begin{array}{cc} (\widehat{B})^*\Pi+(\widehat{D_1})^* \widehat{C}\\ (\widehat{E})^*\Pi
    +(\widehat{D_2})^* 
    \widehat{C} \end{array}\right].
$$
is (pointwise, for each $\varphi \in M(R)$) exponentially stable.

Moreover, from Proposition~\ref{prop2}, it follows that the map
$\varphi\mapsto \Pi(\varphi)$ is continuous on $M(R)$.
In light of the assumption (A1) and \eqref{riccati_1}, we can hence 
conclude that there is a continuous solution $\Pi$ on the maximal ideal
space such that 
\begin{eqnarray*}
\quad \quad 0&=&(\widehat{A^\star})\Pi+\Pi\widehat{A} + (\widehat{C^\star})\widehat{C}
\\
&&-\left[\begin{array}{cc} \Pi \widehat{B}+\widehat{C^\star}\widehat{D_1}& 
\Pi \widehat{E} +\widehat{C^\star} \widehat{D_2}\end{array}\right]
\left[\begin{array}{cc} \widehat{D_1^\star} \widehat{D_1} & \widehat{D_1^\star} \widehat{D_2} \\ 
\widehat{D_2^\star} \widehat{D_1} & \widehat{D_2^\star} \widehat{D_2} -\gamma^2 I \end{array}\right]^{-1} 
\left[\begin{array}{cc} \widehat{B^\star}\Pi+\widehat{D_1^\star} \widehat{C}\\
\widehat{E^\star}\Pi+\widehat{D_2^\star} \widehat{C} \end{array}\right].
\end{eqnarray*}
Finally we will apply Proposition~\ref{prop3}. We have in our case
$s=n^2$, $t=n^2$, the $h_i$:s are the components of $\Pi$ and the
$f_i$:s are the components of $A,A^\star,B,B^\star,C, C^\star ,D_1,
D_1^\star, D_2, D_2^\star, E, E^\star$ (which are totally $r=2(n^2+
nm+pn+pm+p\ell+n\ell) $ in number).  The maps $G_1,\dots, G_{t=n^2}$
are the $n^2$ components of the map $\calG$
$$
\begin{array}{cccc}
(\Theta, U_1,U_2,V_1,V_2,W_1,W_2,X_1,X_2,Y_1,Y_2,Z_1,Z_2)\\
\rotatebox[origin=c]{270}{$\longmapsto$}\;\calG \\
U_2\Theta +\Theta U_1 +W_2 W_1 \phantom{\widehat{D^f}}\\
-\left[\begin{array}{cc}\Theta V_1 +W_2 X_1 & \Theta Z_1+ W_2 Y_1 \end{array}\right]
 \left[\begin{array}{cc} X_2 X_1 & X_2 Y_1 \\ Y_2 X_1 & Y_2 Y_1-\gamma^2 I\end{array}\right]^{-1} 
\left[\begin{array}{cc} V_2 \Theta +X_2 W_1 \\ Z_2 \Theta +Y_2 W_1 \end{array}\right].
\end{array}
$$
(In the above, we have the replacements of 
$$
A, A^\star,B,B^\star,C,C^\star,D_1,D_1^\star,D_2,D_2^\star, E,E^\star
$$ 
by the complex variables which are the components of 
$$
U_1,U_2,V_1,V_2,W_1,W_2,X_1,X_2, Y_1,Y_2, Z_1,Z_2,
$$
respectively.  The replacements of the components of $\Pi$ in the
Riccati equation are by the complex variables which are the components
of $\Theta$.)  Since the set of invertible matrices is open, it follows that 
$$
(X_1,X_2,Y_1,Y_2)\mapsto \left[\begin{array}{cc} X_2 X_1 & X_2 Y_1 \\ Y_2 X_1 & Y_2 Y_1-\gamma^2 I\end{array}\right]^{-1}
$$
is holomorphic on the joint spectrum of $(D_1,D_1^\star,
D_2,D_2^\star)$. Consequently the components $G_1,\dots, G_{n^2}$ of
above map $\calG$ is holomorphic on the joint spectrum. 

There is a continuous solution $\Pi$ on the maximal ideal space such
that for all $k$,
\begin{equation}
\label{eq_riccati_2}
G_k(\Pi, \widehat{A},\widehat{A^\star}, \widehat{B},\widehat{B^\star},\widehat{C}, \widehat{C^\star}, \widehat{D_1}, 
\widehat{D_1^\star},\widehat{D_2}, \widehat{D_2^\star},\widehat{E}, \widehat{E^\star})=0 
\end{equation}
on $M(R)$ (that is, condition \eqref{eq_hayashi_cont_soln} in
Proposition~\ref{prop3} is satisfied). So we now investigate the
Jacobian with respect to the variables in $\Theta$. The Jacobian with
respect to the $\Theta$ variables at the point
$$
\begin{array}{lr}
\Big(\Pi(\varphi), \;\widehat{A}(\varphi), \;
(\widehat{A}(\varphi))^*,\;\widehat{B}(\varphi), \; (\widehat{B}(\varphi))^*,
\;\widehat{C}(\varphi), \; (\widehat{C}(\varphi))^*,\;\\ 
\widehat{D_1}(\varphi), \; (\widehat{D_1}(\varphi))^*,
\;\widehat{D_2}(\varphi), \; (\widehat{D_2}(\varphi))^*,\;\widehat{E}(\varphi), \; (\widehat{E}(\varphi))^*\Big)
\end{array}
$$ 
can be verified to be the following linear transformation $\Lambda$ from
$\mC^{n^2}\rightarrow \mC^{n^2}$:
$$
\Theta \mapsto \Theta \calA_{\textrm{cl}}(\varphi) +(\calA_{\textrm{cl}}(\varphi))^* \Theta.
$$
The set of eigenvalues of $\Lambda$ consists of the numbers 
$$
-(\overline{\lambda}+\mu),
$$
where $\lambda$, $\mu$ belong to the set of eigenvalues of
$\calA_{\textrm{cl}}(\varphi)$; see for example \cite[Proposition~7.2.3]{Ber}.  But
since $\calA_{\textrm{cl}}(\varphi)$ is exponentially stable, all
its eigenvalues have a negative real part.  Hence
$-(\overline{\lambda}+\mu)\neq 0$ for all $\lambda$, $\mu$ belonging
to the set of eigenvalues of the matrix 
$\calA_{\textrm{cl}}(\varphi)$. Consequently, the map $\Lambda$ is invertible from
$\mC^{n^2}$ to $\mC^{n^2}$, and its rank is $n^2=s$. So by
Proposition~\ref{prop3}, there exists a $P\in R^{n\times n}$ such that
$\widehat{P}(\varphi)=\Pi(\varphi)$ for all $\varphi \in M(R)$. From
\eqref{eq_riccati_2}, it follows (using the fact that $R$ is
semisimple) that
\begin{eqnarray*}
\quad \quad 0&=&A^\star P+PA + C^\star C
\\
&&-\left[\begin{array}{cc} P B+C^\star D_1 & 
P E + C^\star D_2 \end{array}\right] 
\left[\begin{array}{cc} D_1^\star D_1 & D_1^\star D_2 \\ 
D_2^\star D_1 & D_2^\star D_2 -\gamma^2 I \end{array}\right]^{-1} 
\left[\begin{array}{cc} B^\star P+ D_1^\star C\\
E^\star P+ D_2^\star C \end{array}\right].
\end{eqnarray*}
From the property possessed by the pointwise solutions $\Pi(\varphi)$
($\varphi \in M(R)$) of the constant complex matricial Riccati
equations \eqref{riccati_1}, we have that for all $\varphi\in M(R)$,
all eigenvalues of $\calA_{\textrm{cl}}(\varphi)$ have a negative
real part. But the set-valued map taking a square complex matrix of
size $n\times n$ to its spectrum (a set of $n$ complex numbers) is a
continuous map; see for example \cite[II,\S 5, Theorem~5.14,
p.118]{Kat}. Since $M(R)$ is compact in the Gelfand topology (the
weak-$\ast$ topology induced on $M(R)$ considered as a subset of
$\calL(R;\mC)$), it follows that
$$
\sup \{ \textrm{Re}(\lambda):\lambda \textrm{ is an eigenvalue of }
\widehat{A_{\textrm{cl}}}(\varphi) \textrm{ for some } \varphi \in M(R)\} <0.
$$
From Proposition~\ref{prop_char_exp_stab}, it follows that
$A_{\textrm{cl}}$ is exponentially stable.  Finally, again by the
property possessed by the pointwise solution $\Pi$, we have that for
all $\varphi\in M(R)$, $\widehat{P}(\varphi)$ is positive
semidefinite.  This completes the proof of Theorem~\ref{main_theorem}.
\end{proof}

We observe that whether or not the assumption (A1) in
Theorem~\ref{main_theorem} holds is intimately related to the choice
of the involution $\cdot^\star$ in the Banach algebra $R$. For some
commutative Banach algebras with involutions, this is automatic,
namely if it is {\em symmetric}; see \cite{CurSas}.

\section{Application to spatially invariant
  systems}\label{Applications}

Finally, we mention the application of our main result to control
problems for spatially invariant systems introduced in
\cite{BamPagDah}. The analysis of spatially invariant systems can be
greatly simplified by taking Fourier transforms, see \cite{BamPagDah},
\cite{CurIftZwa}.  This yields systems described by multiplication
operators with symbols belonging to $L^\infty(\mT)$. The $H^\infty$
control design is to use the feedback $F$, which uses the bounded,
self-adjoint, stabilizing solution $P$ to the $H^\infty$ algebraic
Riccati equation on the Hilbert space $(L^2(\mT))^n$.  For the design
of implementable controllers it is important that the gain operator
have a spatially decaying property (see \cite{BamPagDah}).  This
translates into the mathematical question of when the $H^\infty$
algebraic Riccati equation has a stabilizing solution in a suitable
subalgebra (for example, $L^1(\mT)$ is a subalgebra of
$L^\infty(\mT)\subset {\mathcal L}(L^2(\mT))$).  So the spatially
decaying property now translates into finding suitable subalgebras of
$L^\infty(\mT)$, in particular, the weighted Wiener algebras. In
\cite[p.472-474]{CurSas}, an example of a large class of the Wiener
subalgebra of $L^\infty(\mT)$ having the symmetry property (which
leads to automatic satisfaction of the condition (A1) in our main
result) was given.

\medskip 

\noindent {\bf Acknowledgements:} The author thanks Ruth Curtain for
suggesting to the author the question of investigating if the result
from \cite{CurSas} is valid also for the $H^\infty$ algebraic Riccati
equation, and also for useful discussions on the classical $H^\infty$
problem.


\begin{thebibliography}{99}

\bibitem{BamPagDah}
B. Bamieh, F. Paganini and M.A. Dahleh.
 Distributed control of spatially invariant systems.
 {\em  IEEE Trans. Automatic Control} 47:1091-1107, 2002.

\bibitem{Ber}
D.S. Bernstein. 
{\em Matrix mathematics. Theory, facts, and formulas 
with application to linear systems theory.} 
Princeton University Press, Princeton, NJ, 2005. 

\bibitem{CurIftZwa}
R.F. Curtain, O.V. Iftime and H.J. Zwart. 
System theoretic properties of a class of spatially distributed systems. 
{\em Automatica}, 45:1619-1627, 2009.

\bibitem{CurSas} 
R.F. Curtain and A.J. Sasane. 
On Riccati equations in Banach algebras. 
{\em SIAM Journal on Control and Optimization}, 49:464-475, no. 2, 2011. 




\bibitem{Hay}
M. Hayashi. 
Implicit function theorem for Banach algebras.
{\em Journal of the London Mathematical Society. Second Series.}, 
(2)13:155-161,  no. 1, 1976.

\bibitem{Kat}
T. Kato. 
{\em Perturbation theory for linear operators.} 
Reprint of the 1980 edition. Classics in Mathematics. 
Springer-Verlag, Berlin, 1995.


\bibitem{Ste}
G.W. Stewart.
Error and perturbation bounds for subspaces associated with certain eigenvalue problems.
{\em SIAM Review}, 15:727-764, 1973. 

\bibitem{Sto}
A.A. Stoorvogel. 
{\em The $H_\infty$ control problem: a state space approach.} 
Prentice-Hall, Englewood Cliffs, 1992.

\bibitem{StoSCL}
A. Saberi and A.A. Stoorvogel.
Continuity properties of solutions to $H_2$ and $H_\infty$ Riccati equations. 
{\em Systems and Control Letters}, 27:209-222, no. 4, 1996. 

\end{thebibliography}
\end{document}